\documentclass{amsart}

\usepackage[latin1]{inputenc}
\usepackage{amsfonts}
\usepackage{amsmath}
\usepackage{amsthm}
\usepackage{amssymb}
\usepackage{latexsym}
\usepackage{enumerate}

\newtheorem{theorem}{Theorem}[section]
\newtheorem{lemma}[theorem]{Lemma}

\newtheorem{corollary}[theorem]{Corollary}

\numberwithin{equation}{section}

\begin{document}

\newcommand{\cc}{\mathfrak{c}}
\newcommand{\N}{\mathbb{N}}
\newcommand{\Q}{\mathbb{Q}}
\newcommand{\R}{\mathbb{R}}

\newcommand{\PP}{\mathbb{P}}
\newcommand{\forces}{\Vdash}
\newcommand{\dom}{\text{dom}}
\newcommand{\osc}{\text{osc}}

\title[$\ell_\infty$-sums]
{$\ell_\infty$-sums and the Banach space $\ell_\infty/c_0$}

\author{Christina Brech}
\thanks{The first author was partially supported by FAPESP (2010/12639-1) and Pr\'o-reitoria de Pesquisa USP (10.1.24497.1.2).} 
\address{Departamento de Matem\'atica, Instituto de Matem\'atica e Estat\'\i stica, Universidade de S\~ao Paulo,
Caixa Postal 66281, 05314-970, S\~ao Paulo, Brazil}
\email{christina.brech@gmail.com}

\author{Piotr Koszmider}
\thanks{The second author was partially supported by the National Science Center research grant 2011/01/B/ST1/00657. } 
\address{Institute of Mathematics, Polish Academy of Sciences,
ul. \'Sniadeckich 8,  00-956 Warszawa, Poland}

\email{\texttt{piotr.koszmider@impan.pl}}

\date{}

\begin{abstract} 
This paper is concerned with the isomorphic structure of the Banach space
$\ell_\infty/c_0$ and how it depends on combinatorial tools whose existence is
consistent but not provable from the usual axioms of ZFC.
Our main global result is that it is consistent that
 $\ell_\infty/c_0$ does not have an orthogonal \break $\ell_\infty$-decomposition
 that is, it is not of the form
$\ell_\infty(X)$ for any Banach
space $X$.
The main local result is that it is consistent that $\ell_\infty(c_0(\mathfrak{c}))$ does not embed isomorphically 
into $\ell_\infty/c_0$, where $\mathfrak{c}$ is the cardinality of the continuum,
while $\ell_\infty$ and $c_0(\mathfrak{c})$ always do embed quite canonically.
This should be compared  with  the 
results of Drewnowski and Roberts that under the assumption of the continuum hypothesis
$\ell_\infty/c_0$ is isomorphic to its $\ell_\infty$-sum and in particular it contains an isomorphic copy of all Banach spaces
of the form $\ell_\infty(X)$ for any subspace $X$ of $\ell_\infty/c_0$.

\end{abstract}

\maketitle

\section{introduction}

Drewnowski and Roberts proved in \cite{drewnowski} that, assuming the Continuum Hypothesis (abbreviated CH), the Banach space $\ell_\infty/c_0$ is isomorphic to its $\ell_\infty$-sum denoted  $\ell_\infty(\ell_\infty/c_0)$. They concluded that
under the assumption of CH the Banach space $\ell_\infty/c_0$ is primary, that is, given a decompositon $\ell_\infty/c_0 = 
A\oplus B$, one of the spaces $A$ or $B$ must be isomorphic to $\ell_\infty/c_0$. The proof relies
on the Pe\l czy\'nski decomposition method and on another striking
result from \cite{drewnowski} (not requiring CH) which says that one of the factors $A$ or $B$ as above must contain a complemented subspace isomorphic to
$\ell_\infty/c_0$.  Another conclusion was that  $\ell_\infty(\ell_\infty/c_0)/c_0(\ell_\infty/c_0)$ is  isomorphic to $\ell_\infty/c_0$ under the assumption of CH.

In this paper we show that some of the above statements cannot be proved without some additional 
set theoretic assumptions. Namely, for any cardinal $\kappa \geq \omega_2$, the following statements all hold in the Cohen model obtained by
adding $\kappa$-many Cohen reals to a model of CH ($\cc$ denotes the cardinality of the continuum):
\begin{enumerate}[(a)]
\item $\ell_\infty(c_0(\omega_2))$ does not embed isomorphically  into $\ell_\infty/c_0$,
\item $\ell_\infty(c_0(\cc))$ does not embed isomorphically  into $\ell_\infty/c_0$,
\item $\ell_\infty(\ell_\infty/c_0)$ does not embed isomorphically  into $\ell_\infty/c_0$,
\item $\ell_\infty/c_0$ is not isomorphic to  $\ell_\infty(X)$ for any Banach space $X$,
\item $\ell_\infty(\ell_\infty/c_0)/c_0(\ell_\infty/c_0)$ is not isomorphic to $\ell_\infty/c_0$.
\end{enumerate}

Below we show that (a) easily implies the other statements and so later we will
focus on proving (a). Indeed, (a) implies (b) simply because $\mathfrak{c} \geq \omega_2$ in those models. (c) follows from (b) and the fact that $\ell_\infty/c_0$ contains
an isometric copy of $c_0(\cc)$
(e.g., the closure of the space spanned by the classes of characteristic functions of elements
of a family $\{A_\xi: \xi<\cc\}$ of infinite subsets of $\mathbb{N}$ whose pairwise intersections are finite). 
To conclude (d) from (a), use a result of Rosenthal (see Theorem 7.11 of \cite{biortho}) that if $T:c_0(\Gamma)\rightarrow X$ is a bounded linear operator such that
$|\{\gamma \in \Gamma: |T(\chi_{\{\gamma\}})|>\varepsilon\}| = |\Gamma|$ for some $\varepsilon >0$,
then there is $\Gamma'\subseteq \Gamma$ such that $|\Gamma'|=|\Gamma|$ and $T$ restricted
to $c_0(\Gamma')$ is an isomorphism onto its image; hence
 if $\ell_\infty(X)$ contains $c_0(\omega_2)$, then so does $X$.
Finally (e) follows from (c) alone, because $\ell_\infty(\ell_\infty/c_0)$ embeds isometrically into $\ell_\infty(\ell_\infty/c_0)/c_0(\ell_\infty/c_0)$.
Indeed, consider a partition of $\N$ into pairwise disjoint infinite sets $(A_i: i\in \N)$
and for each $x\in \ell_\infty(\ell_\infty/c_0)$ consider $x'\in \ell_\infty(\ell_\infty/c_0)$ 
such that $x'(n)=x(i)$ if and only if
$n\in A_i$. Note that $T:\ell_\infty(\ell_\infty/c_0)\rightarrow \ell_\infty(\ell_\infty/c_0)$ 
given by $T(x)=x'$ is an isometric embedding. Moreover it gives an isometric embedding
while composed with the quotient map from $\ell_\infty(\ell_\infty/c_0)$ onto
$\ell_\infty(\ell_\infty/c_0)/c_0(\ell_\infty/c_0)$.

We emphasize an interesting phenomenon that follows from the gap which may exist between the number of added Cohen reals and $\omega_2$:
 even when $\mathfrak{c}$ is very large, meaning that $\ell_\infty/c_0$ has large density, still it may not contain an isomorphic copy of  $\ell_\infty(c_0(\omega_2))$ while it always contain quite canonical copies of
both $\ell_\infty$ and
$c_0(\omega_2)$. 

It remains unknown if $\ell_\infty/c_0$ is primary in the above models and in general
if the primariness of $\ell_\infty/c_0$ can be proved without additional set theoretic assumptions.
It would also be interesing to conclude the above statements in a more axiomatic way as in
\cite{stevo} or \cite{krupski}.

Another problem mentioned in \cite{drewnowski} remains open as well (including in the Cohen model), namely if
$\ell_\infty/c_0$ has the Schroeder-Bernstein property, that is if there exists a complemented
subspace $X$ of $\ell_\infty/c_0$, nonisomorphic to $\ell_\infty/c_0$ but
which contains a complemented isomorphic
copy of $\ell_\infty/c_0$. The Pe\l czy\'nski decomposition method and the existence of an isomorphism between
$\ell_\infty/c_0$ and $\ell_\infty(\ell_\infty/c_0)$ implies that $\ell_\infty/c_0$ 
has the Schroeder-Bernstein property assuming CH. On the other hand, the nonprimariness of 
$\ell_\infty/c_0$ would imply that it does not have the Schroeder-Bernstein property
as observed in \cite{drewnowski}. It could be noted that after the
first example of a Banach space without the Schroeder-Bernstein property was given in \cite{gowers},
an example of the form $C(K)$, like all the spaces considered in this paper, was constructed as well (see \cite{sb}).

Our results (a) - (c) can also be seen in a different light. It is well-known that 
assuming CH the space $\ell_\infty/c_0$ is isometrically universal for all Banach spaces of 
density not bigger than $\cc$. It has been proved by the authors in \cite{universal} 
that this is not the case in the Cohen model, even in the isomorphic sense. 
The results (a) - (c) show that $\ell_\infty(c_0(\cc))$ or 
 $\ell_\infty(\ell_\infty/c_0)$ can be added to a recently growing list of Banach spaces that consistently
do not embed into $\ell_\infty/c_0$, see \cite{ug}, \cite{krupski}, \cite{stevo} or section 3 of \cite{universal}.
A new feature of the examples provided in this paper is that they are neither
obtained from a well-ordering of the continuum nor a generically constructed object like those in the above mentioned
papers.

In Section 2 we present some consequences of the assumption that $\ell_\infty/c_0$ contains an isomorphic copy of $\ell_\infty(c_0(\lambda))$ for some uncountable cardinal $\lambda$ and Section 3 contains the key forcing lemma (Lemma \ref{maindisjointlemma}), whose proof is inspired by the proof of A. Dow of Theorem 4.5 of \cite{alan}
that the boundary of a zero set in $\N^*$ is not a retract of $\N^*$ in the Cohen model.

The undefined notation of the paper is fairly standard. Undefined notions related
to set theory and independence proofs can be found in \cite{kunen} and those
related to Banach spaces in \cite{fabian}.

Let us now introduce some particular notation concerning the spaces we consider here.
Given $A \subseteq \N$, let us denote by $[A]$ the corresponding equivalence class in $\wp(\N)/Fin$, by $A^*$ the corresponding clopen set of $\beta \N$ and by $[A]^*$ the clopen set of $\N^*=\beta\N\setminus \N$ corresponding to $[A]$.

Given $x \in \ell_\infty$, let us denote by $[x]$ the corresponding equivalence class in $\ell_\infty/c_0$. We will use the isometries $\ell_\infty \equiv C(\beta \N)$ and $\ell_\infty/c_0 \equiv C(\N^ *)$ and identify each bounded sequence with its extension to $\beta \N$ and each class $y = [x]$ of bounded sequences in $\ell_\infty/c_0$ with the restriction to $\N^ *$ of an extension of $x$ to $\beta \N$. 

For $m,n \in \N$, $\alpha, \beta \in \lambda$ and $\sigma \in \lambda^A$ for $A\subseteq \N$, let 
$$1_{n,\alpha}(m)(\beta) = \left\{\begin{array}{ll}
1 & \text{ if }(n,\alpha)=(m,\beta)\\
0 & \text{ otherwise,}
\end{array}\right.$$
$$1_\sigma(m)(\beta) = \left\{\begin{array}{ll}
1 & \text{ if }(m,\beta) \in \sigma\\
0 & \text{ otherwise,}
\end{array}\right.$$
and notice that $1_{n,\alpha}, 1_\sigma \in \ell_\infty(c_0(\lambda))$ and they
can be thought of as the characteristic functions of
$\{(n,\alpha)\}$ and of the graph of $\sigma$ respectively.

Some of the problems addressed in this paper were considered in \cite{grzech}
under different set-theoretic assumptions. 
Unfortunately the forthcomming paper announced there  which was to contain
the proofs of the statements instead of sketches of the proofs
has not appeared as far as now. Also the statements and arguments outlined in \cite{grzech} on page 303
concerning the Cohen model contradict our results.

\section{Facts on isomorphic embeddings of $l_\infty(c_0(\lambda))$ into $\ell_\infty/c_0$}

\begin{lemma}\label{locallyconstant}
Suppose $y\in \ell_\infty/c_0 \setminus \{0\}$ and $A\subseteq \N$ is infinite. Then there is an infinite $B\subseteq A$ and $r \in \mathbb{R}$ such that $|r| \geq \frac{\Vert y|[A]^* \Vert}{2}$ and $y|[B]^* \equiv r$.
\end{lemma}
\begin{proof} 
Let $x = (x_n)_{n \in \N} \in \ell_\infty$ be such that $y=[x]$. Since $B'=\{n \in A: |x_n|> \frac{\Vert y|[A]^* \Vert}{2}\}$ is infinite and $\{x_n: n \in B'\}$ is bounded, there is an infinite $B \subseteq B'$ such that $(x_n)_{n \in B}$ converges to some $r \in \R \setminus\{0\}$. Notice that $|r| \geq \frac{\Vert y|[A]^* \Vert}{2}$ and $y|[B]^* \equiv r$.
\end{proof}

\begin{theorem}\label{bigrosenthal} 
Assume $\lambda$ is an uncountable cardinal and $T:\ell_\infty(c_0(\lambda))\rightarrow \ell_\infty/c_0$ is an isomorphic embedding. Then there is $X \in[\lambda]^{\lambda}$ and for each 
$(n, \alpha) \in \N \times X$ there is an infinite set $E_{n, \alpha}\subseteq \N$ and $r_{n, \alpha}\in\R$
such that 
$$|r_{n, \alpha}| \geq \frac{\Vert T(1_{n,\alpha})\Vert}{2}$$
and if $\sigma \in \lambda^\N$ is an injective function  such that $Im(\sigma) \subseteq X$, then for
 all $(n, \alpha) \in \N \times X$,
$$T(1_\sigma)|[E_{n,\alpha}]^* \equiv \left\{\begin{array}{ll}
r_{n,\alpha} & \text{ if }(n, \alpha) \in \sigma\\
0  & \text{ if }\alpha \notin Im(\sigma). \end{array}\right.$$
\end{theorem}
\begin{proof}
For each $n \in \N$ and each $\alpha \in \lambda$, by Lemma \ref{locallyconstant} there is $r_{n,\alpha}\in \R$ and an infinite set $E_{n,\alpha}'\subseteq \N$ such that 
$$|r_{n, \alpha}| \geq \frac{\Vert T(1_{n,\alpha})\Vert}{2}$$
and $T(1_{n,\alpha})|[E_{n,\alpha}']^* \equiv r_{n,\alpha}$.
\vskip 6pt
\noindent{\bf Claim:} \textit{For every $(n,\alpha) \in \N \times \lambda$, there is a countable set $X(n,\alpha)\subseteq \lambda$ and an infinite set $E_{n,\alpha}\subseteq^* E_{n,\alpha}'$ such that whenever $\sigma \in \lambda^A$
for some nonempty $A\subseteq \N$ is such that $Im(\sigma)\cap X(n, \alpha) = \emptyset$, then we have}
$$T(1_\sigma)|[E_{n,\alpha}]^*\equiv 0.$$
\vskip 6pt
\noindent \textit{Proof of the claim:}
If the claim does not hold, let $(n,\alpha)\in\N\times \lambda$ 
for which the claim fails. We will carry out certain transfinite inductive
construction of length $\omega_1$ which will lead to a contradiction. We will construct for each $\xi < \omega_1$ an infinite set $F_\xi\subseteq \N$, $r_\xi\in\R \setminus \{0\}$ and $\sigma_\xi\in \lambda^{A_\xi}$ for some nonempty $A_\xi\subseteq \N$ such that 
\begin{enumerate}
\item $F_\eta\subseteq^* F_\xi\subseteq^* E_{n,\alpha}'$ for all $\xi<\eta<\omega_1$,
\item $\sigma_\xi \cap \sigma_\eta = \emptyset$ for all $\xi<\eta<\omega_1$,
\item $T(1_{\sigma_\xi})|[F_{\xi}]^*\equiv r_\xi$ for all $\xi < \omega_1$.
\end{enumerate}

Given $\xi < \omega_1$, suppose we have already constructed infinite sets $(F_\eta)_{\eta < \xi} \subseteq \wp(\N)$, $(r_\eta)_{\eta < \xi} \subseteq \R \setminus \{0\}$ and $\sigma_\eta\in \lambda^{A_\eta}$
for some nonempty $A_\eta\subseteq\N$ and all $\eta<\xi$ as above. Let $F_\xi' \subseteq \N$ be an infinite set such that $F_\xi' \subseteq^* F_\eta$ for every $\eta < \xi$. Since $\Lambda = \bigcup\{Im(\sigma_\eta): \eta <\xi\}$ is a countable subset of $\lambda$, by our 
hypothesis there is $\sigma_\xi \in \lambda^{A_\xi}$ 
for some nonempty $A_\xi\subseteq\N$ such that $Im(\sigma_\xi) \cap  \Lambda = \emptyset$ and
$$T(1_{\sigma_\xi})|[F_\xi']^* \not\equiv 0$$
and using Lemma \ref{locallyconstant} find $F_\xi \subseteq F_\xi'$ infinite and $r_\xi \in \R \setminus \{0\}$ such that
$$T(1_{\sigma_\xi})|[F_\xi]^* \equiv r_\xi.$$
This concludes the inductive construction of objects satisfying (1), (2) and (3).

We can now find some $\varepsilon >0$ for which $R_\varepsilon = \{\xi < \omega_1: |r_\xi| \geq \varepsilon\}$ is infinite (uncountable, actually) and splitting $R_\varepsilon$ into two sets, we may assume without loss of generality that either $r_\xi \geq \varepsilon$ for every $\xi \in R_\varepsilon$ or $-r_\xi \geq \varepsilon$ for every $\xi \in R_\varepsilon$.

Fix $m \in \N$ such that $m \cdot \varepsilon > \Vert T \Vert$. Choose $\xi_1<... <\xi_m$ in $R_\varepsilon$ and notice that $|\sum_{i\leq m}r_{\xi_i}| \geq m \cdot \varepsilon >||T||$.

Since the $\sigma_{\xi_i}$'s are pairwise disjoint, we get that 
$$||\sum_{i\leq m} 1_{\sigma_{\xi_i}}||=1$$ 
but
$$||T(\sum_{i\leq m} 1_{\sigma_{\xi_i}})||
\geq ||T(\sum_{i\leq m} 1_{\sigma_{\xi_i}})|[F_{\xi_m}]^*||
\geq |\sum_{i\leq m}r_{\xi_i}|>||T||,$$
which is a contradiction and completes the proof of the claim.

\vskip 6pt

For each $\alpha \in \lambda$, let $X(\alpha) = \bigcup_{n \in \N} X(n,\alpha)$ and notice that $X(\alpha)$ is a countable subset of $\lambda$ such that for every $n \in \N$, there is an infinite set $E_{n,\alpha}\subseteq E_{n,\alpha}'$ such that whenever $\sigma \in \lambda^\N$ and $Im(\sigma) \cap X(\alpha) = \emptyset$, then we have
\begin{equation}\label{eq1}
T(1_\sigma)|[E_{n,\alpha}]^*\equiv 0.
\end{equation}

Now apply the Hajnal free-set lemma (Lemma 19.1, \cite{hajnalhamburger}) to obtain $X \subseteq \lambda$
of cardinality $\lambda$ such that $X(\alpha)\cap X\subseteq \{\alpha\}$ for each $\alpha\in X$. This implies that for distinct $\alpha, \beta \in X$, $\alpha \notin X(\beta)$.

Given $\sigma \in \lambda^\N$ which is injective and such that $Im(\sigma) \subseteq X$, 
notice that for distinct $n, n'\in \N$, $\sigma(n) \notin X(\sigma(n'))$,  which guarantees that 
$Im(\sigma \setminus\{(n, \sigma(n))\}) \cap X(\sigma(n)) = \emptyset$. 

Assume $\sigma\in\lambda^\N$ is
injective and  for each $(n, \alpha) \in \omega \times X$ let us consider the two cases. If $(n, \alpha) \in \sigma$,
then
$$T(1_\sigma)|[E_{n,\alpha}]^* = T(1_{n,\alpha})|[E_{n,\alpha}]^* + T(1_{\sigma \setminus \{(n,\alpha)\}})|[E_{n,\alpha}]^* \equiv r_{n,\alpha},$$
where the last equality follows from (\ref{eq1}) and the choice of $E_{n, \alpha}$.

If $\alpha \notin Im(\sigma)$, it follows from (\ref{eq1}) that
$$T(1_\sigma)|[E_{n,\alpha}]^* \equiv 0.$$
\end{proof}

Although the above theorem is sufficient for our applications, let us note that it has the following more
elegant version.

\begin{corollary}\label{embedding}
Assume $\lambda$ is an uncountable cardinal and $T:\ell_\infty(c_0(\lambda))\rightarrow \ell_\infty/c_0$ is an isomorphic embedding. Then there is an isomorphic embedding $T':\ell_\infty(c_0(\lambda))\rightarrow \ell_\infty/c_0$ and for each $(n, \alpha) \in \N \times \lambda$ there is an infinite set $E_{n, \alpha}\subseteq \N$ and $r_{n, \alpha}\in\R$
such that 
$$|r_{n, \alpha}| \geq \frac{\Vert T'(1_{n,\alpha})\Vert}{2}$$
and for all $(n,\alpha)\in \N \times \lambda$, if $\sigma \in \lambda^\N$, then
$$T'(1_\sigma)|[E_{n,\alpha}]^* \equiv \left\{\begin{array}{ll}
r_{n,\alpha} & \text{ if }(n, \alpha) \in \sigma\\
0 & \text{ otherwise}.\end{array}\right.$$
\end{corollary}
\begin{proof}
Let $X\subseteq \lambda$ of cardinality $\lambda$ and an infinite set $E_{n, \alpha}\subseteq \N$ and $r_{n, \alpha}\in\R$ for each $(n, \alpha) \in \N \times X$ be as in Theorem \ref{bigrosenthal}.

Let $(X_n)_{n \in \N}$ be a partition of $X$ into countably many sets of cardinality $\lambda$ and enumerate each $X_n$ as $X_n = \{\gamma^n_\beta: \beta < \lambda\}$.

Define $S: \ell_\infty(c_0(\lambda)) \rightarrow \ell_\infty (c_0(\lambda))$ by 
$$S(f)(n)(\beta) = \left\{\begin{array}{ll}
f(n)(\alpha) & \text{ if }\beta = \gamma^n_\alpha\\
0 & \text{ otherwise}.\end{array}\right.$$
Notice that $S$ has the following properties:
\begin{itemize}
\item $S$ is an isometric embedding.
\item If $\sigma \in \lambda^\N$, 
then $S(1_\sigma) = 1_{s(\sigma)}$ where $s(\sigma) =
 \{(n, \gamma^n_\alpha): (n, \alpha) \in \sigma\}$, 
so that $s(\sigma) \in \lambda^\N$ is injective, 
 $Im(s(\sigma)) \subseteq X$ and 
$\gamma_{n,\alpha}\in Im(s(\sigma))$ if and only if $(n,\alpha) \in \sigma$.
\end{itemize}


Let $T' = T \circ S$, $E_{n, \alpha}' = E_{n,\gamma^n_\alpha}$ and $r_{n,\alpha}' = r_{n, \gamma^n_\alpha}$. 
Given $\alpha < \lambda$ and $n\in \N$, 
$$|r_{n, \alpha}'| = |r_{n,\gamma^n_\alpha}| \geq \frac{\Vert T(1_{n,\gamma^n_\alpha})\Vert}{2} = \frac{\Vert T(S(1_{n,\alpha}))\Vert}{2}= \frac{\Vert T'(1_{n,\alpha})\Vert}{2}.$$
Also, given any $\sigma \in \lambda^\N$ and $(n, \alpha) \in \N \times \lambda$, we have:
\begin{itemize}
\item if $(n,\alpha) \in \sigma$, then
$$T'(1_\sigma)|[E_{n,\alpha}']^* 
= T(S(1_\sigma))|[E_{n,\gamma^n_\alpha}]^* = r_{n,\gamma^n_\alpha} = r_{n,\alpha}',$$
\item if $(n,\alpha) \notin \sigma$, then $\gamma^n_\alpha \notin Im(s(\sigma))$ so that
$$T'(1_\sigma)|[E_{n,\alpha}']^* = T(1_{s(\sigma)})|[E_{n,\gamma^n_\alpha}]^* = 0.$$
\end{itemize}
This concludes the proof that $T'$ is the isomorphic embedding with the required properties.
\end{proof}

\section{Forcing argument}

The next lemma still holds if we replace $\omega_2$ by any regular cardinal $\lambda$ with $\omega_2 \leq \lambda \leq \kappa$. To simplify the notation we state it in this weaker form, which is sufficient for our purposes.

\begin{lemma}\label{maindisjointlemma} 
Let $V$ be a model of CH, $\kappa \geq \omega_2$ and $\PP=Fn_{<\omega}(\kappa, 2)$. In $V^{\PP}$, 
if $(E_{n,\alpha}: (n,\alpha) \in \N \times \omega_2)$ are infinite subsets of $\N$
and for each $\sigma \in \omega_2^\N$, $B_{\sigma}$ is a subset of $\N$ such that 
$$\forall (n,\alpha) \in \sigma \quad E_{n,\alpha}\subseteq^* B_{\sigma},$$
then there is a pairwise disjoint subset $\Sigma\subset \omega_2^\N$ of cardinality $\omega_2$
such that $\{B_\sigma: \sigma\in \Sigma\}$ has the finite intersection property, that is, for every $\sigma_1, \dots, \sigma_m \in \Sigma$, $B_{\sigma_1} \cap \dots \cap B_{\sigma_m}$ is infinite.
\end{lemma}

\begin{proof} 
In $V$, for $(n, \alpha) \in \N \times \omega_2$ let $\dot{E}_{n,\alpha}$ be a nice-name for an infinite subset of $\N$.

For each $n\in\N$ and $\alpha\in \omega_2$, let $S_{n,\alpha} = supp (\dot{E}_{n, \alpha})$, which are countable subsets of $\kappa$ since $\PP$ is ccc.

By CH and the $\Delta$-system lemma, we may find pairwise disjoint $(A_n)_{n \in \omega} \subseteq [\omega_2]^{\omega_2}$ such that
\begin{itemize}
\item for each $n \in \N$, $(S_{n, \alpha})_{\alpha \in A_n}$ is a $\Delta$-system with root $\Delta_n$.
\end{itemize}

Let $\Delta = \bigcup_{n \in \N} \Delta_n$ and since this is a countable set, by a further thinning out of each $A_n$, we may assume that 
\begin{itemize}
\item for every $\alpha \in A_n$, $\Delta \cap (S_{n,\alpha} \setminus \Delta_n) = \emptyset$, i.e. $\Delta \cap S_{n,\alpha} = \Delta_n$. 
\end{itemize}

We may also assume that for each $\alpha < \beta$ in $A_n$ there is a bijection $\pi_{n, \alpha,\beta}: S_{n,\alpha} \rightarrow S_{n, \beta}$ such that
\begin{itemize}
\item $\pi_{n, \alpha,\beta}|_{\Delta_n} = id$,
\item $\pi_{n, \alpha,\beta} (\dot{E}_{n, \alpha}) = \dot{E}_{n, \beta}$ (here $\pi_{n, \alpha,\beta}$ denotes the automorphism of $\PP$ obtained by lifting $\pi_{n, \alpha,\beta}$).
\end{itemize}

Inductively choose, for $\xi < \omega_2$, functions $\sigma_\xi \in \omega_2^\N$ such that:
\begin{itemize}
\item $\sigma_\xi(n) \in A_n$ for each $n\in\N$,
\item for all $\xi < \eta < \omega_2$ and all distinct $n,m \in \N$, 
$$(S_{n, \sigma_\xi(n)} \setminus \Delta_n) \cap (S_{m, \sigma_\eta(m)}\setminus \Delta_m) = \emptyset,$$
\item for all $\xi < \eta < \omega_2$, $\sup_{n \in \omega} \sigma_\xi(n) < \sigma_\eta(0)$, 
so that $\sigma_\xi \cap \sigma_\eta = \emptyset$. 
\end{itemize}

For each $\xi <\omega_2$, let $\dot{B}_{\xi}$ be a name for a subset of $\N$ 
as in the hypothesis of the lemma, that is, such that 
$$\PP \Vdash \forall (n,\alpha) \in \check{\sigma}_\xi \quad \dot{E}_{n,\alpha}\subseteq^* \dot{B}_{\sigma_\xi}$$
and let $\dot{h}_{\xi}$ be a nice-name such that
$$\PP \Vdash \dot{h}_{\xi}: \N \rightarrow \N \text{ is such that } \forall (n,\alpha) \in 
\check{\sigma}_\xi \quad \dot{E}_{n, \alpha} \setminus \dot{h}_{\xi}(n) \subseteq \dot{B}_{\sigma_\xi}.$$
Let $R_\xi = supp (\dot{h}_{\xi})$ and $S_\xi = \bigcup_{n \in \mathbb{N}} (S_{n, \sigma_\xi(n)} \setminus \Delta_n)$ and notice that the $S_\xi$'s are pairwise disjoint countable subsets of $\kappa$. By the $\Delta$-system lemma 
and further thinning out there is $A\subseteq\omega_2$ of cardinality $\omega_2$ 
such that 
\begin{itemize}
\item $(R_{\xi})_{\xi\in A}$ is a $\Delta$-system with root $R$,
\item for all $\xi \in A$, $\Delta \cap (R_{\xi} \setminus R) = \emptyset$,
\item for all distinct $\xi, \eta \in A$, $S_\xi \cap R_{\eta} = \emptyset$,
\end{itemize}
where the last property can be achieved by Hajnal's free set lemma (Lemma 19.1, \cite{hajnalhamburger}).

Fix $m \in \N$ and $\xi_1<... <\xi_m$ from $A$ and let us prove that 
$$\mathbb{P} \Vdash \dot{B}_{\sigma_{\xi_1}} \cap \dots \cap \dot{B}_{\sigma_{\xi_m}} \text{ is infinite,}$$
which would give that $\{B_{\sigma_\xi}:\xi\in A\}$ has the finite intersection property.
Otherwise, there are $p \in \mathbb{P}$ and $l \in \mathbb{N}$ such that 
$$p \Vdash \dot{B}_{\sigma_{\xi_1}} \cap \dots \cap \dot{B}_{\sigma_{\xi_m}} \subseteq \check l.$$

Given $n \in \N$ we say that $q\in \mathbb{P}$ is $n$-symmetric if 
$$\pi_{n, \sigma_{\xi_i}(n), \sigma_{\xi_j}(n)}(q|_{S_{n, \sigma_{\xi_i}(n)}}) = q|_{S_{n, \sigma_{\xi_j}(n)}}
 \text{ for all }1 \leq i <j\leq m.$$

Fix $n\in \N$ such that $dom(p) \cap S_{n,\sigma_{\xi_i}(n)} = \emptyset$  for all $1 \leq i \leq m$ and 
notice that $p$ is $n$-symmetric. Let us find $q \leq p$ which is $n$-symmetric and $k_1, \dots, k_m \in \N$ such that $q \Vdash \dot{h}_{\xi_i}(\check n) = \check k_i$. To do so, we will construct conditions
$p_m \leq p_{m-1} \leq \dots \leq p_1 \leq p$ and $k_1, \dots, k_m \in \N$ such that each $p_i$ is $n$-symmetric and $p_i \Vdash \dot{h}_{\xi_i}(\check n) = \check{k}_i$. Let $p_0 = p$.

Given $1 \leq i \leq m$, let $q_i \leq p_{i-1}$ and $k_i \in \N$ be such that $q_i \Vdash \dot{h}_{\xi_i}(\check n) = \check{k}_i$ and $dom(q_i)\setminus dom(p_{i-1})\subseteq R_{\xi_i}$. Let 
$$p_i = q_i \cup \bigcup_{1 \leq j \leq m} \pi_{n, \sigma_{\xi_i}(n),\sigma_{\xi_j}(n)}(q_i|_{S_{n, \sigma_{\xi_i}(n)}}).$$
Notice that $p_i \in \mathbb{P}$, $p_i$ is $n$-symmetric and, since $p_i \leq q_i$, $p_i \Vdash \dot{h}_{\xi_i}(\check n) = \check{k}_i$, as we wanted.

Now $p_m$ is an $n$-symmetric condition such that 
$$p_m \Vdash \forall 1 \leq i \leq m \quad \dot{h}_{\xi_i}(\check n) = \check{k}_i.$$

Since $\mathbb{P}$ forces that $\dot{E}_{n, \sigma_{\xi_1}(n)}$ is infinite, let $r_1 \leq p_m$ and $n_0 > \max\{k_1, \dots, k_m, l\}$ be such that $r_1 \Vdash n_0 \in \dot{E}_{n, \sigma_{\xi_1}(n)}$ and $dom(r_1) \setminus dom(p_m) \subseteq S_{n, \sigma_{\xi_1}(n)}$. 
Let
$$r = r_1 \cup \bigcup_{2 \leq j \leq m} \pi_{n, \sigma_{\xi_1}(n),\sigma_{\xi_j}(n)}(r_1)$$
and notice that $r \in \mathbb{P}$, $r \leq p$ and
$$r\forces {\check n}_0 \in (\dot{E}_{n,\xi_1}\setminus \dot{h}_{\xi_1}(n))\cap \dots \cap (\dot{E}_{n,\xi_m}\setminus \dot{h}_{\xi_m}(n)),$$
so that 
$$r\forces {\check n}_0 \in \dot{B}_{\sigma_{\xi_1}}\cap \dots \cap \dot{B}_{\sigma_{\xi_m}},$$
which contradicts our assumption since $n_0>l$. This concludes the proof.
\end{proof}

\begin{theorem}\label{maintheorem}
Let $V$ be a model of CH, $\kappa \geq \omega_2$ and $\PP=Fn_{<\omega}(\kappa, 2)$. In $V^{\PP}$
there is no isomorphic embedding $T:\ell_\infty(c_0(\omega_2))\rightarrow \ell_\infty/c_0$.
\end{theorem}
\begin{proof}
We work in $V^\PP$. Let $\varepsilon >0$ be such that $\frac{1}{\Vert T^{-1} \Vert} > \varepsilon$. 

By Corollary \ref{embedding} we may assume that for each $(n, \alpha) \in \N \times \omega_2$ there is an infinite set $E_{n, \alpha}\subseteq \N$ and $r_{n, \alpha}\in\R$
such that 
\begin{equation}\label{eq*}
|r_{n, \alpha}| \geq \frac{\Vert T(1_{n,\alpha})\Vert}{2} \geq \frac{1}{2 \cdot \Vert T^{-1}\Vert} > \frac{\varepsilon}{2}
\end{equation}
and for all $n\in \N$ and all $\alpha\in \omega_2$, if $\sigma \in \omega_2^\N$, then
\begin{equation}\label{eq**}
T(1_\sigma)|[E_{n,\alpha}]^* \equiv \left\{\begin{array}{l}
r_{n,\alpha}, \text{ if }(n, \alpha) \in \sigma\\
0, \text{ otherwise}.\end{array}\right.
\end{equation}


For each $\sigma \in \omega_2^\N$, fix any representative $x_{\sigma} 
\in C(\beta \N)$ of $T(1_\sigma)$ and let
$$B_{\sigma} = \{k \in \N: |x_{\sigma}(k)| > \frac{\varepsilon}{4}\}.$$

Then, we get that $B_{\sigma}$'s are as in the hypothesis of Lemma \ref{maindisjointlemma}. 

Given $m \in \N$ such that $\Vert T \Vert < m \cdot \frac{\varepsilon}{4}$, by Lemma \ref{maindisjointlemma} there are pairwise disjoint $\sigma_1, \dots, \sigma_{2m} \in \omega_2^\N$ such that 
$$B=B_{\sigma_1}\cap... \cap B_{\sigma_{2m}} \text{ is infinite.}$$

Given $u \in [B]^*$, let $1 \leq j_1 < \dots < j_m \leq 2m$
 be such that $T(1_{\sigma_{j_i}})(u)$ are either all positive or all negative. Then
$$|T(\sum_{i=1}^m 1_{\sigma_{j_i}})(u)| \geq m \cdot \frac{\varepsilon}{4} > \Vert T \Vert,$$
which contradicts the fact that
$$\Vert \sum_{i=1}^m 1_{\sigma_{j_i}} \Vert = 1$$
and concludes the proof.
\end{proof}

Note that apparently we did not use in the above proof the entire strength of Corollary \ref{embedding}, namely
we do not use the fact that $T(1_\sigma)|E_{n,\alpha}\equiv0$ when $(n,\alpha)\not\in \sigma$.
However this is used within the proof of Corollary \ref{embedding} to conclude that
$T(1_\sigma)|E_{n,\alpha}\equiv r_{n,\alpha}$ when $(n,\alpha)\in \sigma$.

\bibliographystyle{amsplain}

\end{document}